\documentclass[reqno,10pt]{amsart}
\numberwithin{equation}{section}
\usepackage{pdfsync}
\usepackage{amsmath}
\usepackage{amsthm}
\usepackage{epsfig}
\usepackage{psfrag}
\usepackage{graphicx}
\usepackage{graphpap,latexsym,epsf}
\usepackage{color}
\usepackage{amssymb,mathrsfs,enumerate}
\usepackage{eucal}

\newcommand{\R}{\Rz}

\newcommand{\Rz}{\mathbb{R}}

\newcommand{\eps}{\varepsilon}

\newcommand{\down}{\downarrow}

\newcommand{\dualoperator}

  \def\calC{{\mathcal C}}
 \def\calE{{\mathcal E}} \def\calF{{\mathcal F}}
\def\calG{{\mathcal G}}  
 \def\calK{{\mathcal K}} \def\calL{{\mathcal L}}
 \def\calN{{\mathcal N}} 
  
\def\calS{{\mathcal S}}  
 \def\calW{{\mathcal W}} 
 
\def\rmd{{\mathrm d}}  
\def\rmg{{\mathrm g}}

  \def\rmC{{\mathrm C}}
\def\rmD{{\mathrm D}}  \def\rmF{{\mathrm F}}
\def\rmG{{\mathrm G}}

\def\dd{\;\!\mathrm{d}} 

\newcommand{\pairing}[4]{ \sideset{_{ #1 }}{_{ #2 }}  {\mathop{\langle #3 , #4
\rangle}}}
\newcommand{\paired}[2]{ \sideset{_{ X^* }}{_{ X }}  {\mathop{\langle #1 , #2
\rangle}}}

\newcommand{\nchi}{{\raise.2ex\hbox{$\chi$}}}

\definecolor{ddcyan}{rgb}{0,0.1,0.9}
\definecolor{ddmagenta}{rgb}{0.8,0,0.8}
\definecolor{orange}{rgb}{0.6,0.2,0}
\definecolor{dred}{rgb}{.8,0,0}
\definecolor{ddgreen}{rgb}{0,0.4,0.4}
\definecolor{vgreen}{rgb}{0.1,0.5,0.2}

\newcommand{\piecewiseConstant}[2]{\overline{#1}_{\kern-1pt#2}}

\newcommand{\sft}{\mathsf{t}}
\newcommand{\sfz}{\mathsf{z}}
\newcommand{\sfx}{\mathsf{x}}

 \def\trait #1 #2 #3 {\vrule width #1pt height #2pt depth #3pt}

 \def\fin{\hfill
         \trait .3 5 0
         \trait 5 .3 0
         \kern-5pt
         \trait 5 5 -4.7
         \trait 0.3 5 0
 \medskip}
\newtheorem{theorem}{Theorem}[section]
\newtheorem{remark}[theorem]{Remark}
\newtheorem{corollary}[theorem]{Corollary}
\newtheorem{definition}[theorem]{Definition}
\newtheorem{proposition}[theorem]{Proposition}

\newtheorem{assumption}[theorem]{Assumption}
\newtheorem{lemma}[theorem]{Lemma}
\newtheorem{example}[theorem]{Example}
\newcommand{\berin}{} 
\newcommand{\erin}{} 

\newcommand{\rbe}{}
\newcommand{\rend}{}

\renewcommand{\rmF}{f}
\renewcommand{\rmG}{g}
\newcommand{\Y}{\mathsf Y}
\newcommand{\X}{\mathsf X}
\renewcommand{\H}{\mathsf H}
\newcommand{\Z}{\mathsf Z}
\renewcommand{\SS}{\mathsf S}
\newcommand{\NN}{\mathsf N}
\newcommand{\RR}{\mathsf R}
\renewcommand{\dim}{\mathrm{dim}}
\newcommand{\codim}{\mathrm{codim}}
\newcommand{\ind}{\mathrm{ind}}
\newcommand{\scrE}{\mathscr E}

\newcommand{\Oomega}{X}
\begin{document}
\title{On the transversality conditions and their genericity}
\date{\today}

\author[V.~Agostiniani]{Virginia Agostiniani}
\address{Mathematical Institute, University of Oxford, Woodstock road,
Oxford, OX2 6GG, UK}
\email{\ttfamily virginia.agostiniani\,@\,maths.ox.ac.uk}

\author[R.~Rossi]{Riccarda Rossi}
\address{DICATAM-Sezione di Matematica, Universit\`a di
Brescia, via Valotti 9, I--25133 Brescia, Italy.}
\email{\ttfamily riccarda.rossi\,@\,ing.unibs.it}

\author[G.~Savar\'e]{Giuseppe Savar\'e}
\address{Dipartimento di Matematica ``F.~Casorati'', Universit\`a di Pavia.
Via Ferrata, 1 -- 27100 Pavia, Italy.}
\email{\ttfamily giuseppe.savare\,@\,unipv.it}

\thanks{V.~A. has received funding from the European Research Council under the
European Union's Seventh Framework Programme (FP7/2007-2013) / ERC grant agreement ${\rm n^o}$ 291053.
R.~R. and G.~S. have been partially supported by  a MIUR-PRIN 2012 grant for the  project ``Calculus of Variations''.}

\begin{abstract}
In this note we review some results on the \emph{transversality
conditions} for a smooth Fredholm map $\rmF: \X \times (0,T) \to \Y$
between two Banach spaces $\X,\Y$. These conditions are well-known
in the realm of bifurcation theory and commonly accepted as
``generic''. Here we show that under the transversality assumptions
the  sections $\calC(t)=\{x:\rmF(x,t)=0\}$ of the zero set of $f$
are discrete for every $t\in (0,T)$ and we discuss a somehow
explicit family of perturbations of $\rmF$ along which
transversality holds up to a residual set.

The application of these results to the case
when $\rmF$ is the $\X$-differential of
a time-dependent energy functional $\scrE:\X\times (0,T)\to \R$ and
$\calC(t)$ is the set of the critical points of $\calE$
provides  \berin the motivation  and   the
main example of this \erin  paper.
\end{abstract}

\maketitle

\section{Introduction}\label{s:intro}

\noindent Let $\X,\Y$ be a couple of Banach spaces and let $\rmF:
\Oomega \times (0,T) \to \Y$ be a $\rmC^2$-map  defined in the open
set $\Oomega\subset \X$.

In this note, we investigate the so-called  \emph{transversality conditions}
for the set $\calS$ of the singular points of $\rmF$, i.e.~the subset
of the zero set $\calC$ of
$\rmF$ where its (partial) differential
$\rmD_x\rmF$ with respect to $x\in \X$ is not invertible:
\begin{equation}
  \label{eq:2}
  \calC:=\Big\{(x,t)\in \Oomega\times (0,T):\rmF(x,t)=0\Big\},
  \quad
    \calS:=\Big\{(x,t)\in\calC:\rmD_x \rmF (x,t) \text{ is non invertible}\Big\}.
\end{equation}
In particular we will focus on two important features
related to transversality.
The first one concerns the topology of
the sections $\calC(t):=\Big\{x\in \Oomega:(x,t)\in \calC\Big\}$ of
$\calC$: it is not difficult to show that
for every $t\in (0,T)$ the sets $\calC(t)$ are discrete, so that each point in $\calC(t)$ is isolated.
A second property we will discuss in some detail \rbe is \rend the
\emph{generic} character of transversality.

We first illustrate \rbe our results,  \rend and their role for
applications, in the simpler finite dimensional setting.
\subsection*{The transversality conditions in the finite-dimensional case}
\label{s:1}

\noindent
Let us consider the case when $\X=\Y=\R^n$.
 \berin The symbol $\langle \cdot,\cdot\rangle$ represents \erin the scalar product in
 $\R^n$, and $\mathbb{M}^{k \times m}$ is the space of matrixes with
 $k$ rows and $m$ columns.
If $K\in \mathbb M^{k\times m}$, we denote by $\NN(K)$ its null
space in $\R^k$, \rbe by \rend $\RR(K)$ its range in $\R^m$, \rbe and by $K^*$ its transposed. \rend 

\begin{definition}[Transversality conditions in $\R^n$]
  \label{ass:tran}
  We say that \berin $\rmF$ satisfies at a point $(x_0,t_0) \in \calS$
 the transversality conditions if\erin
\begin{enumerate}[{\rm (T1)}]
\item
  $\mathrm{dim}(\NN (\rmD_x \rmF (x_0,t_0)))=1$
  $($and
  therefore
  $\mathrm{dim}(\NN(\rmD_x \rmF (x_0,t_0)^*))=1${}$)$.
\item
  If $0\neq w^*\in \NN (\rmD_x \rmF (x_0,t_0)^*)$ then
  $\langle \partial_t \rmF (x_0,t_0),  w^* \rangle \neq 0$.
\item If $0\neq v\in \NN (\rmD_x \rmF (x_0,t_0))$
  then $\langle \rmD_x^2 \rmF (x_0,t_0)[v,v], w^* \rangle
  \neq 0 $.
\end{enumerate}
We say that $\rmF$ satisfies the transversality conditions if they hold at
every point of $\calS$.
\end{definition}
 While referring to  Definition  \ref{ass:tran-infinity} \rbe ahead
 \rend
 for the precise statement of the transversality conditions
 in an infinite-dimensional setting,
  let us point out that, there
  it is crucial, as well as natural,  to require in addition
that $\rmF(\cdot,t)$ is a Fredholm map of index $0$ for every $t \in (0,T)$.

\rbe Conditions (T2) and (T3) \rend look particularly simple when
$n=1$: in this case they read 
\[
\partial_t \rmF(x_0,t_0)  \neq 0, \qquad \partial_{xx}^2 \rmF(x_0,t_0)  \neq 0.
\]
%

\subsection*{Singular perturbation of gradient flows and the
  structure of $\calC(t)$.}
Our motivation for getting further insight on the topology of $\calC(t)$
stems from the study of the  limit as $\eps \down 0$ of the
differential equation
\begin{subequations}
  \label{subeq:gflow}
  \begin{equation}
    \label{gflow-intro}
    \eps x'(t) + f(x(t),t) =0, \quad  t  \in (0,T),
  \end{equation}
  in particular in the case of a gradient flow, when
  \begin{equation}
    \label{eq:10}
    f=\rmD_x\mathscr E,\quad
    \mathscr{E} : \Oomega \times (0,T) \to \R
    \quad\text{ is a $\rmC^3$ functional.}
  \end{equation}
\end{subequations}
The study of the \emph{singular perturbation} problem
(\ref{subeq:gflow}a,b) was carried out \rbe in finite dimension \rend
in \cite{zanini}, where it was shown that
a family $\{x_\eps\}_\eps$ of solutions to \eqref{gflow-intro}
 converges to a piecewise regular curve obtained by connecting  smooth branches of solutions to the equilibrium equation
 \begin{equation}
 \label{equil-intro}
 f(x(t),t) =0, \quad \text{i.e. }  x(t)\in \calC(t),\quad   t  \in (0,T),
\end{equation}
by means
  of heteroclinic solutions of the gradient flow
 \begin{equation}
 \label{pure-gflow}
 \theta'(s) + f(\theta(s),t) =0,\quad
 \lim_{s\to\pm\infty}\theta(s)=x_\pm(t)\quad \text{at every jump point $t$
   for the limit curve $x$.}
  \end{equation}
  The analysis
 in \cite{zanini}
  hinges on the assumption that  $ \rmF= \rmD_x \mathscr{E}$
   satifies the transversality conditions at  its singular points (which are the degenerate critical points of
     $\mathscr{E}$).
     In addition, a
   crucial role is played by the condition that
   at every time $\underbar t\in (0,T)$ there exists at most
   one degenerate critical point $(\underbar x,\underbar t)$ in $\calS$,
   $\rmD^2_x \mathscr E$ is positive semidefinite at $(\underbar x,\underbar t)$
   and the (unique) heteroclinic starting
   from $\underbar x$ ends in a local minimum of $\mathscr
   E(\cdot,\underbar t)$.
   This allows  the author to
   tackle this singular limit  by means of
  refined
  techniques from bifurcation theory, see also \cite{virginia}, where the quasistatic limit of a \emph{second-order} system 
  was
  addressed in  the finite-dimensional case.

  In the forthcoming
   paper \cite{ars},  we develop
   a different \emph{variational} approach to the limit $\eps \to 0$
   of \eqref{gflow-intro},
   even in the setting of an
   infinite-dimensional space $\X$.
   Under general coercivity assumptions on $\mathscr E$,
   we prove that, up to a subsequence, any family $\{x_\eps\}_\eps $
   of solutions
   to \eqref{gflow-intro} pointwise converges to a curve $x$ fulfilling
   the equilibrium equation
    \eqref{equil-intro} and
    a variational jump condition.
   In fact, the crucial property that lies at the core of the analysis  in   \cite{ars}
   is that $\calC(t)$ is discrete for every $t \in (0,T)$.

The next result shows that if \berin $\rmF$ \erin satisfies the transversality conditions
then \berin $\calC(t)$ \erin is a discrete set for every $t\in (0,T)$.
\begin{theorem}
\label{prop:isolated} If \berin $\rmF$ \erin satisfies the transversality conditions
then for every $t \in (0,T)$ the set \berin $\calC(t)$ \erin only consists
of isolated points.
\end{theorem}
We present here the simple proof in the finite dimensional case,
postponing the discussion of the general case to the next section.
\begin{proof}
Let us fix \berin $x_0\in\calC(t_0)$. If $\rmD_x\rmF(x_0,t_0)$ \erin  is invertible,
then the thesis follows immediately by the implicit function theorem.
It is thus sufficient to consider the case  when $(x_0,t_0)\in\calS$.

 Let us first show that \berin $\NN(\rmd\rmF(x_0,t_0))$ has
dimension $1$ so that the differential of $\rmF$ \erin at
$(x_0,t_0)$ has rank $n$ and is surjective. If \berin
$(v,\alpha)\in\NN(\rmd\rmF(x_0,t_0))$, \erin then taking the duality
with $w^*$ as in \berin (T2) \erin we get
\begin{align*}
  0&=\langle\rmD_x\rmF(x_0,t_0)[v],w^*\rangle+ \alpha
  \langle\partial_t\rmF(x_0,t_0),w^*\rangle =
  \langle\rmD_x\rmF(x_0,t_0)^*[w^*],v\rangle+ \alpha
  \langle\partial_t\rmF(x_0,t_0),w^*\rangle
  \\&=\alpha
  \langle\partial_t\rmF(x_0,t_0),w^*\rangle,
\end{align*}
so that $\alpha=0$ by (T2). Choosing now an arbitrary vector $w^*\in
\R^n$ we get $v\in\NN(\rmD_x\rmF(x_0,t_0))$, which has dimension $1$
thanks to (T1).

Since 
the differential $\rmd\rmF(x_0,t_0)$ has rank $n$,
the implicit function theorem implies
that there exists an open neighborhood $U$ of $(x_0,t_0)$
and a smooth curve
\[
\gamma : (-\eps,\eps) \to \Oomega \times (0,T), \qquad s \mapsto
\gamma(s) =(\sfx(s),\sft(s)),
\]
such that
\begin{equation}\label{gamma-props}
\begin{gathered}
(\sfx(0),\sft(0)) = (x_0,t_0), \quad
\dot\gamma(s)\neq 0\quad \text{in }(-\eps,\eps),
\quad
\berin\calC\erin\cap U=\gamma((-\eps,\eps)).
\end{gathered}
\end{equation}
Differentiating w.r.t\  $s$ the identity
$\berin\rmF\erin(\sfx(s),\sft(s))=0$ we then obtain
\begin{equation}
\label{differentiation} \rmD_x \rmF (\sfx(s),\sft(s))[\dot{\sfx}(s)]
+ \dot{\sft}(s) \partial_t \rmF (\sfx(s),\sft(s)) =0 \quad \forall\,
s \in (-\eps,\eps).
\end{equation}
Multiplying \eqref{differentiation} at $s=0$  by $w^* $ we \rbe find
\rend
\[
\begin{aligned}
0&=\langle \rmD_x \rmF (x_0,t_0)[\dot{\sfx}(0)], w^* \rangle +
\dot{\sft}(0) \langle
\partial_t \rmF (x_0,t_0), w^* \rangle\\
 & = \langle \rmD_x \rmF (x_0,t_0)^*[ w^*], \dot{\sfx}(0) \rangle +
\dot{\sft}(0) \langle
\partial_t \rmF (x_0,t_0), w^* \rangle
=0+
\dot{\sft}(0) \langle
\partial_t \rmF (x_0,t_0), w^* \rangle.
\end{aligned}
\]
Since $\langle\partial_t \rmF (x_0,t_0), w^* \rangle \neq 0$, we then have
$\dot{\sft}(0) =0$, \berin which implies $\dot{\sfx}(0)\neq0$ by the second of (\ref{gamma-props}).
Evaluating (\ref{differentiation}) at $s=0$
we therefore \erin get $\rmD_x \rmF (x_0,t_0)[\dot{\sfx}(0)]=0$,
i.e.~$\dot\sfx(0)\in\berin\NN(\rmD_x\rmF(x_0,t_0))\erin$.

Differentiating w.r.t\ $s$ twice, we have
\[
\begin{aligned}
\rmD_x^2 \rmF (\sfx(s),\sft(s))[\dot{\sfx}(s)]^2  & + 2 \dot{\sft}(s) \partial_t \rmD_x \rmF (\sfx(s),\sft(s))[\dot{\sfx}(s)]
 +  \rmD_x \rmF (\sfx(s),\sft(s))[\ddot{\sfx}(s)]  \\ & + \dot{\sft}(s)^2 \partial_t^2 \rmF (\sfx(s),\sft(s))+ \ddot{\sft}(s) \partial_t\rmF (\sfx(s),\sft(s))=0\quad \forall\, s \in (-\eps,\eps).
\end{aligned}
\]
Multiplying the above relation at $s=0$ by $w^*$
and taking into account that
$\dot\sfx(0)\in\berin\NN(\rmD_x\rmF(x_0,t_0))\setminus \{0\}\erin$ and $\dot{\sft}(0) =0$, we then conclude
\[
\ddot{\sft}(0)  \langle \partial_t\rmF (x_0,t_0), w^* \rangle
=\langle \rmD_x^2 \rmF (x_0,t_0)[\dot\sfx (0),\dot \sfx(0) ], w^*
\rangle\neq 0,
\]
thanks to \berin (T3) from Definition \ref{ass:tran}. \erin
The above relation then gives
\[
\ddot{\sft}(0)  \neq0.
\]
This yields the existence of a neighborhood $V$
of $(x_0,t_0)$ such that  $V\cap\berin\calC\erin(t_0)$ contains only the point $x_0$,
since a simple Taylor expansion of $\sft$ around $s=0$ shows
that $\sft(s)\neq\berin t_0\erin$ in a sufficiently small neighborhood of $0$.
\end{proof}
%
%
%
%
\subsection*{Genericity of the transversality conditions}
\newcommand{\notilde}{\relax}
 As the results from \cite{zanini}, \cite{virginia}, and \cite{ars}
reveal their crucial role,
 it is interesting to investigate to what extent
the transversality conditions
can be assumed to hold for  ``generic'' vector fields.  In the realm of bifurcation theory,
this is commonly accepted, see e.g.\ \cite{guck-holmes, haragus,
  vanderbauwhede}.
Nonetheless,  it is not always easy to find, in the literature, an
explicit result stating the precise meaning of genericity,
especially in the infinite-dimensional setting (see \cite{Sot} for a
\emph{finite-dimensional} version of the genericity result).

We address this issue
in Section \ref{s:3}. Its main result, Theorem \ref{thm:genericity_infty}, states that
 a certain class of perturbations of $\rmF$
 satisfy the transversality conditions up to a \emph{meagre} subset
(or, in finite dimension, up to a Lebesgue negligible set).
 The key idea at the basis of our proof is that the
 transversality conditions 
 are equivalent to the property
 that the ``augmented" operator 
 \begin{equation}
 \label{aug-g}
\notilde\rmG: \Oomega \times (0,T) \times  (\X\setminus \{0\})  \to
\Y \times \Y \qquad \rmG(x,t,v) := (\notilde\rmF(x,t), \rmD_x
\notilde\rmF(x,t)[v]),
\end{equation}
has  $(0,0)$  as a regular value, namely that  the total differential $\dd \notilde g$ is surjective at each point of
$\notilde g^{-1} (0,0)$. This is proved in Lemma \ref{l:G-(4)} below. We will  combine this fact
with a classical result from the paper \cite{Saut-Temam79}. Theorem 1.1 therein
indeed provides
 conditions
ensuring that,
 up to a meagre set, the perturbations of a given operator have $(0,0)$ as a regular value.

 \rbe Our finite-dimensional genericity result reads \rend
%
%
%
\begin{theorem}\label{thm:genericity_finite}
Let $\rmF \in\rmC^3 (\Oomega \times (0,T); \R^n)$. There exists a
meagre (or with zero Lebesgue measure) set $N\subset
\R^n\times\mathbb{M}^{n \times n}$ such that for every $(y,K)\in(
\R^n\times\mathbb{M}^{n \times n})\setminus\berin N\erin$
 the map
\begin{equation}\label{eq:perturbation_F-finite}
\tilde\rmF:\Oomega\times(0,T) \to \R^n,
\qquad \tilde\rmF(x,t):= \rmF(x,t)+y+ K[x]
\end{equation}
satisfies the transversality conditions \berin of Definition \ref{ass:tran}.\erin
\end{theorem}
The proof of this result, based on the classical Sard's theorem,
and the motivation for the enhanced regularity requirement on $f$,
are postponed to \berin Section \ref{ss:4.1}, Remark \ref{rem:fdimproof}. \erin


\paragraph{\bf Plan of the paper.}  In Section \rbe \ref{ss:1.4}, we analyze the transversality conditions in
the setting of an infinite-dimensional Banach space $\X$. We prove
that they imply that the zeroes of $f$ are isolated (cf.\ Thm.\
 \ref{prop:isolated_infty}).
In Section \ref{ss:2.2} \rend we also obtain the characterization of
the transversality conditions in terms of the operator $g$
\eqref{aug-g}, which is the milestone for our genericity result,
Theorem \ref{thm:genericity_infty},  proved in Sec.\ \ref{ss:4.1}.
Finally, in Section \ref{ss:4.2}, in view of the application to the
 singular perturbation problem \rbe  \eqref{gflow-intro}--\eqref{eq:10} \rend
 we specialize the discussion and our results to the case when
$f=\rmD_x \mathscr{E}$, with  $\mathscr{E} : \X \times (0,T) \to \R$ a smooth energy functional.

\section{The transversality conditions in the infinite-dimensional case}
\label{ss:1.4}

\noindent
In the infinite-dimensional context, it is natural to study the  transversality conditions
for the class of vector fields $\rmF: \Oomega \times (0,T) \to \Y$ which
are \emph{Fredholm maps} of index $0$ between two Banach spaces $\X,\Y$
(see below for a definition).
We prove Lemma \ref{l:crucial-infinite-dim} which
provides a characterization of the first two transversality conditions in terms of the surjectivity of the operator \berin $\dd \rmF$. \erin
We rely on this in the proof of the main result of this section, viz.\ Theorem \ref{prop:isolated_infty},
which states that the transversality conditions imply that \berin  the  zeroes \erin of $f$  are isolated.
Finally, in Section \ref{ss:2.2}, we obtain a characterization of the \emph{full} set of
transversality conditions in Lemma \ref{l:G-(4)}.
We will exploit this result in order to investigate their genericity in Section \ref{s:3}.

\subsection{Assumptions and preliminary results}
\label{s:prelims}

\paragraph{\bf Notation and preliminary definitions.}
Let $\X$ be a Banach space;
we shall denote by
$\| \cdot \|_\X$ its norm, by $\X^*$ its dual,
and by $\pairing{\X^*}{\X}{\cdot}{\cdot}$
the duality pairing between \berin$\X^*$ and $\X$\erin.

Given two Banach spaces $\X$ and $\Y$,
we denote by $\mathcal{L}(\X;\Y)$ the space of linear bounded
operators from $\X$ to $\Y$ and by $\mathcal L^k(\X;\Y)$ the
space of the continuous $k$-linear forms from $\X^k$ to $\Y$.
If $L\in \mathcal L(\X;\Y)$ we denote by $L^*$ its adjoint operator
in $\berin\mathcal L\erin(\Y^*;\X^*)$.
We also set
\begin{equation}
  \label{eq:4}
  \text{kernel of $L$}:\
  \NN(L)=\big\{x\in \X:Lx=0\big\},\quad
  \text{range of $L$}:\
  \RR(L)=\big\{Lx:x\in \X\big\}.
\end{equation}
A subspace $\mathsf R\subset \Y$ has finite codimension
if there exists a finite-dimensional space $\mathsf S\subset \Y$ such
that $\mathsf R+\mathsf S=\mathsf Y$.
In this case $\mathsf R$ is closed and
we can always choose $\mathsf S$ so that $\mathsf S\cap
\mathsf R=\{0\}$; the codimension of $\mathsf R$ is then defined as
$\mathrm{codim}(\mathsf R):=\mathrm{dim}(\mathsf S)$.
\berin Moreover, $\RR(L^*)$ is closed \erin
and we have the adjoint relations
\begin{equation}
  \label{eq:5}
  \RR(L)^\perp=\NN(L^*),\qquad
  \NN(L)^\perp=\RR(L^*).
\end{equation}
We recall that $L\in\mathcal L(\X;\Y)$
is a \emph{Fredholm operator} if
\[
\mathrm{dim}(\NN( L)),\ \mathrm{codim}(\RR(L))\quad\text{ are finite.}
\]
In particular the range of a Fredholm operator is closed.
%
The \emph{index} of the operator $L$ is defined as
\begin{equation}
  \label{eq:6}
  \ind(L):=\berin\dim(\NN(L))\erin-\codim(\RR(L)).
\end{equation}
For Fredholm operators \eqref{eq:5} yields
\begin{equation}
  \label{eq:7}
  \ind(L)=0\quad\Longleftrightarrow\quad
  \dim(\NN(L))=\dim(\NN(L^*)),\quad
  \codim(\RR(L))=\codim(\RR(L^*)).
\end{equation}
The stability theorem shows that the collection of all Fredholm
operators is open in $\mathcal L(\X;\Y)$ and the index $\ind$ is a locally
constant function.
Let $\Oomega\subset \X$ be a connected open subset of $\X$.
A map $f\in \rmC^1(\Oomega;\Y)$ is a Fredholm map if for every
$x\in \Oomega$ the differential
 $\rmD f(x)\in \mathcal L(\X,\Y)$ is a Fredholm operator.
 The \emph{index} of $f$ is defined as the index of $\rmD f(x)$
 for some $x\in \Oomega$. By the stability theorem
and the connectedness of $\Oomega$
this definition is independent of $x$.

In this infinite-dimensional context, our basic assumption on the vector field
$ \rmF : \Oomega \times (0,T) \to \Y$ is the following.
\begin{assumption}
\label{F_ass1_infty} We require that $ \rmF \in \rmC^2   (\Oomega \times
(0,T); \Y)$, and that
\begin{equation}
\label{fredhom-zero} \text{$\rmF(\cdot,t): \Oomega \to \Y$ is a Fredholm map of
index $0$ for every $t\in(0,T)$. }
\end{equation}
\end{assumption}
We shall denote by $\rmD_x \rmF : \Oomega \times (0,T) \to
\mathcal{L} (\X; \Y)$ and by $\partial_t \rmF : \Oomega \times (0,T)
\to \Y$ the partial G\^{a}teau derivatives of $\rmF$,  whereas $\rmd
f$ is the total differential of $f$. As usual, the second order
(partial) differential $\rmD^2_x f:\Oomega\times (0,T)\to \mathcal
L^2(\X;\Y)$ \rbe at  a point $(x,t)$  \rend is identified with its canonical  bilinear form. As in
the previous section we set
\begin{equation}
  \label{eq:8}
  \calC:=\Big\{(x,t)\in \Oomega\times (0,T):
  \rmF(x,t)=0\Big\},
  \quad
  \calS:=\Big\{(x,t)\in \calC:\rmD_x\rmF(x,t)\ \text{is not
    invertible}\Big\},
\end{equation}
and we denote by $\calC(t)$ and $\calS(t)$ their sections at the time
$t\in (0,T)$.
Let us now state the  infinite-dimensional
version of \berin Definition \erin \ref{ass:tran}.
\begin{definition}[Transversality conditions in the infinite-dimensional case]
  \label{ass:tran-infinity}  Let $\rmF$ comply with Assumption
  \ref{F_ass1_infty}.
  We say that \berin $\rmF$ satisfies at a point $(x_0,t_0) \in \calS$
 the transversality conditions if\erin
\begin{enumerate}[\rm (T1)]
\item
  $\mathrm{dim}(\NN (\rmD_x \rmF (x_0,t_0)))=1$;
\item
  If $0\neq \ell^*\in \NN(\rmD_x \rmF(x_0,t_0)^*)$
  then $\berin\pairing{\Y^*}{\Y}{\ell^{*}}{\partial_t \rmF (x_0,t_0)}\erin\neq 0$;
\item
  If $0\neq v\in \NN(\rmD_x \rmF(x_0,t_0))$ then
  $\berin\pairing{\Y^*}{\Y}{\ell^{*}}{\rmD_x^2 \rmF (x_0,t_0)[v,v] }\erin\neq 0$.
\end{enumerate}
If the above properties hold at every $(x_0,t_0)\in\calS$,
we say that $\rmF$ satisfies the transversality conditions.
\end{definition}

\begin{remark}\label{rmk:index_0_index_1}
{\rm
Thanks to \berin \eqref{eq:7} and \eqref{fredhom-zero}\erin,
the above condition \berin (T1) \erin yields that $\NN (\rmD_x \rmF(x_0,t_0)^*)$
has also dimension $1$.
Note that Assumption \ref{F_ass1_infty} implies that  $\rmF: \Oomega \times (0,T) \to \Y $  is a Fredholm map of
index $1$ \berin with respect to the variable $(x,t)$. Indeed, \erin
\begin{equation}
  \label{eq:14}
  \rmd \rmF(x,t)[v,\tau]=\rmD_x \rmF(x,t)[v]+\tau\partial_t \rmF(x,t);
\end{equation}
if $\partial_t\rmF(x,t)\in \RR\big(\rmD_x\rmF(x,t)\big)$ then
$\codim\big(\RR(\rmd\rmF(x,t))\big)=
\codim\big(\RR(\rmD_x\rmF(x,t))\big)$ but
$\dim\big(\NN(\rmd\rmF(x,t))\big)=\dim\big(\NN(\rmD_x\rmF(x,t))\big)+1$.
On the other hand, if
$\partial_t\rmF(x,t)\not\in \RR\big(\rmD_x\berin\rmF\erin(x,t)\big) $
then
$\codim\big(\RR(\rmd\rmF(x,t))\big)=
\codim\big(\RR(\rmD_x\rmF(x,t))\big)-1$ and
$\dim\big(\NN(\rmd\rmF(x,t))\big)=\dim\big(\NN(\rmD_x\rmF(x,t))\big)$.
}
\end{remark}

Let us now get further insight into the transversality conditions \berin (T1) and (T2). \erin
\begin{lemma}
\label{l:crucial-infinite-dim}
Suppose that $\rmF : \Oomega \times (0,T) \to \Y$ satisfies Assumption \ref{F_ass1_infty} and let
 $(x_0,t_0)\in\calS$ be fixed.
Then conditions
\berin$(T1)$ and $(T2)$ of Definition \ref{ass:tran-infinity} \erin  hold
if and only if $\rmd \rmF (x_0,t_0)$ is onto.
\end{lemma}

\begin{proof}
Let us suppose that $\rmd\rmF(x_0,t_0)$ is onto.
The inequality $\mathrm{dim}(\NN(\rmD_x \rmF (x_0,t_0))) \leq 1$
follows immediately by Remark \ref{rmk:index_0_index_1} and
\berin from the fact that $\codim\big(\rmd\rmF(x_0,t_0)\big)=0$, because
$\NN\big(\rmD_x\rmF(x_0,t_0)\big){\times}\{0\}\subseteq
\NN\big(\rmd\rmF(x_0,t_0)$\big). \erin
Since $(x_0,t_0)\in\calS$, we conclude that (T1) holds.

Suppose now that $0\neq \ell^*\in \NN\big(\rmD_x \rmF
(x_0,t_0))^*\big) $ and let $\xi \in \Y $ \berin be such that $\pairing{\Y^*}{\Y}{ \ell^*}{ \xi}\neq 0$.\erin
Since $\rmd\rmF(x_0,t_0)$ is onto, there exists $(v,\tau) \in \X\times \R$ such that
\[
\xi  = \rmD_x \rmF (x_0,t_0)[v] + \tau \partial_ t\rmF (x_0,t_0).
\]
Therefore
\begin{equation}
\label{ci-serve-infty}
\begin{aligned}
0\neq  &\berin\pairing{\Y^*}{\Y}{ \ell^*}{\xi}=
\pairing{\Y^*}{\Y}{\ell^*}{ \rmD_x \rmF(x_0,t_0)[v]} +
\tau\pairing{\Y^*}{\Y}{\ell^*}{\partial_ t\rmF(x_0,t_0)}\erin\\
&\quad = \berin\pairing{\X^*}{\X}{ \rmD_x \rmF (x_0,t_0)^*[\ell^*] }{v} +
\tau\pairing{\Y^*}{\Y}{\ell^*}{\partial_ t\rmF (x_0,t_0) } = 0+
\tau\pairing{\Y^*}{\Y}{\ell^*}{\partial_ t\rmF (x_0,t_0) },\erin
\end{aligned}
\end{equation}
so that $\berin\pairing{\Y^*}{\Y}{\ell^*}{\partial_ t\rmF (x_0,t_0) }\neq0\erin$.

In order to prove the converse implication,
let us suppose that (T1-2) hold
and let $\ell^*\in \RR\big(\rmd
F(x_0,t_0)\big)^\perp\subset \Y^*.$ Since for every $v\in \X$ and
$\tau\in \R$
\begin{displaymath}
 \berin\pairing{\Y^*}{\Y}{\ell^*}{\rmD_x\rmF(x_0,t_0)[v]+\tau\partial_t\rmF(x_0,t_0)}=0\erin,
\end{displaymath}
by choosing $\tau=0$
we immediately get $\ell^*\in \RR(\rmD_x\rmF(x_0,t_0))^\perp=\NN(\rmD_x
\rmF(x_0,t_0)^*)$;
property (T2) then yields $\ell^*=0$,
so that $\rmd\rmF(x_0,t_0)$ is onto.
\end{proof}

\subsection{The transversality conditions imply that the zeroes are isolated}
\label{ss:3.2}

\noindent
We are now in the position to  state and prove  the  analogue of Theorem \ref{prop:isolated}:

\begin{theorem}\label{prop:isolated_infty}
Suppose that $\rmF$ satisfies Assumptions \ref{F_ass1_infty} and \berin the transversality conditions
of Definition  \ref{ass:tran-infinity}. \erin
Then, for every $t\in(0,T)$ the set
$
\mathcal C(t):=\left\{x\in \X\,:\,\rmF(x,t)=0\right\}
$
consists of isolated points.
\end{theorem}

\noindent
The proof follows the same outline as for Theorem \ref{prop:isolated}.
Indeed, we first of all observe that, for $t_0 \in (0,T)$ fixed,  \berin \emph{non-singular} \erin  points $x$
(i.e., such that  $\berin\NN\erin(\rmD_x\rmF(x,t_0))$ is trivial) are isolated.
In the \berin singular \erin case, after some preliminary discussion  we again apply the implicit function theorem
in order to deduce that, in  a neighborhood of  $(x_0,t_0)$ with $x_0$  \berin (singular) zero, \erin
the zero set of $\rmF$ is the graph of a suitable function.
This  allows us to exploit the transversality conditions and deduce by the same arguments
as in the finite-dimensional case, that $x_0$ is isolated.

\begin{proof}
 Let us fix $(x_0,t_0)\in\calC$ and consider first the case when
$\NN(\rmD_x \rmF (x_0,t_0))=\{0\}$.
Since $\rmD_x\rmF(x_0,t_0)$ is a Fredholm operator of index $0$,
$\rmD_x\rmF(x_0,t_0)$ is invertible
and we can apply  the infinite-dimensional version of the implicit
function theorem (see e.g.\ \cite[Theorem 2.3]{Amb_Pro})
and conclude that $x_0$ is isolated in $\calC(t_0)$.

Now suppose that $\NN_0:=\NN\big(\rmD_x \rmF (x_0,t_0)\big)$ is
non-trivial. Then, by \berin (T1) of Definition \ref{ass:tran-infinity}, \erin
$\NN_0=\mathrm{span}(n_0)$ for some $n_0 \in \X\setminus\{0\}$,
so that we can write $\X=\Z+\NN_0$, where
$\Z$ is the
topological supplement of $\NN_0$ in $\X$. Thus, every $x\in \X$
can be uniquely expressed as $x=z+sn_0 $, with $z\in
\Z$ and $s\in \R$.
In particular, it is not restrictive to assume $x_0=z_0+n_0$, $z_0\in \Z$.

Let us now consider the function $\tilde\rmF(z,s,t):=\berin\rmF\erin(z+sn_0,t)$
defined in a neighborhood of $(z_0,1,t_0)$ in $\Z\times \R\times \R$.
In analogy to the proof of Theorem \ref{prop:isolated}, we now show that
the differential of the function $(z,t)\to\tilde\rmF(z,1,t)$ at the point $(z_0,t_0)$, denoted by
$\rmD_{(z,t)}\tilde \rmF(z_0,1,t_0)$, is invertible.
Let us note first that by construction and Lemma \ref{l:crucial-infinite-dim}
\begin{equation}\label{su_Y}
\RR\big(\rmD_{(z,t)}\tilde \rmF(z_0,1,t_0)\big) =
\RR\big(\rmd\rmF(x_0,t_0)\big)=\Y.
\end{equation}
To check injectivity of $\rmD_{(z,t)}\tilde\rmF(z_0,1,t_0)$,
suppose that $(v,\tau)\in \NN\big(\rmD_{(z,t)}\tilde\rmF(z_0,1,t_0)\big)
\berin\subseteq\Z\times \R\erin$, i.e.
$$
0=\rmD_z\tilde \rmF(z_0,1,t_0)[v]+\tau \partial_t\tilde \rmF(x_0,1,t_0)=
\rmD_x\rmF(x_0,t_0)[v]+\tau\partial_t \rmF(x_0,t_0).
$$
Taking the duality with $\berin0\neq\erin\ell^*\in \NN\big(\rmD_x\rmF(x_0,t_0)^*\big)=
\RR\big(\rmD_x\rmF(x_0,t_0)\big)^\perp$ and recalling \berin (T2) \erin we get $\tau=0$.
Then $\rmD_x\rmF(x_0,t_0)[v]=0$ so that $v\in \Z\cap \NN_0=\{0\}$.
This concludes the proof of the
injectivity of $\rmD_{(z,t)}\tilde\rmF(z_0,1,t_0)$.

Since $\rmD_{(z,t)}\tilde \rmF(z_0,1,t_0)$ is invertible,
we can now apply the implicit function theorem:
there exists  a neighborhood  $U$ of $(x_0,t_0)$,
an open interval $I\ni 0$, and
a $\rmC^2$-curve $(\sfz,\sft):I \to \Z\times (0,T)$
such that setting $\gamma(s):=(\sfz(s)+(1+s)n_0,\sft(s))$ we have
\begin{equation}\label{funz_imp_infty}
  \gamma(0)=(x_0,t_0),\quad
  U\cap \calC=\gamma(I).
\end{equation}
Differentiating the identity $\rmF(\gamma(s))=0$ w.r.t. $s$ and evaluating at $s=0$ we obtain
\begin{equation}\label{diffe}
0=\rmD_x\rmF(x_0,t_0)[\dot\sfz(0)+n_0]+\dot\sft(0)\partial_t\rmF(x_0,t_0).
\end{equation}
Testing by
$0\neq \ell^*\in \NN\big(\rmD_x \rmF
(x_0,t_0)^*\big)=\RR\big(\rmD\rmF(x_0,t_0)\big)^\perp$ and
recalling \berin (T2) \erin as well as $\dot \sfz\in \Z$,
we then obtain
\begin{equation}\label{deriv_1_0}
\berin\dot\sft\erin(0)=0,\quad \dot\sfz(0)=0.
\end{equation}
A further differentiation
w.r.t. $s$ gives at $s=0$
\[
0=\rmD_x^2\rmF(x_0,t_0)[n_0,n_0]+\rmD_x\rmF(x_0,t_0)[\ddot\sfz(0)]
+\ddot\sft(0)\partial_t\rmF(x_0,t_0),
\]
also in view of (\ref{deriv_1_0}).
Testing the last expression by $\ell^*$ as before, we have
\[
0=\pairing\Y{\Y^*}{\rmD_x\rmF(x_0,t_0)[n_0,n_0]}{\ell^*}+\ddot\sft(0)\pairing\Y{\Y^*}{
\partial_t\rmF(x_0,t_0)}{ \ell^*}.
\]
From (T3) we deduce that $\berin\ddot\sft\erin(0)\neq 0$
and we conclude that
$0$ is the only solution of the equation $\sft(s)=
t_0$
in a small neighborhood of $0$, so that $x_0$ is an isolated point in
$\calC(t_0)$ by \eqref{funz_imp_infty}.
\end{proof}
 It is interesting to note that when $f$ is \emph{analytic} then
conditions (T1-T2) of Definition \ref{ass:tran-infinity} are still
sufficient to provide a useful property of $\calC(t)$.
\begin{theorem}
  Suppose that $f$ is an analytic map satisfying Assumptions {\upshape \ref{F_ass1_infty}} and
  {\upshape (T1-T2)} of Definition \ref{ass:tran-infinity}, and
  that any connected component of $\calC(t)$ is compact.
  Then the set $\calC(t)$ is the disjoint union of a discrete set
  and \rbe of \rend
  an analytic manifold of dimension $1$, whose connected components
  are compact curves. Moreover, for every connected curve $C\subset \calC(t)$
  there exist $\eps>0$ and an open neighborhood $V$
  such that $\calC(s)\cap V=\emptyset $ for every $s\in (0,T)$ with $0<|t-s|<\eps$.
\end{theorem}
\begin{proof}
  Let us keep the same notation of the proof of
  Theorem \ref{prop:isolated_infty} and let $x_0$ be a singular
  point of $\calC(t_0)$.
  Since $f$ is analytic, the curve
  $\gamma$ parameterizing the set $\calC$ in the neighborhood $U$ of
  the singular point $(x_0,t_0)$ is analytic.
  If $x_0$ is an accumulation point of $\calC(t_0)$ then
  the $\sft$-component of $\gamma$ takes the value $t_0$ in a set accumulating at
  $0$,
  so that $\sft$ has to be identically constant. It follows that
  $\gamma(I)$ is contained in $\calC(t_0)$ so that we conclude that
  the connected component $C$ of $\calC(t_0)$ containing $x_0$ is a (non-degenerate)
  compact analytic curve.
  The last assertion follows by the fact that
  every point of $C\subset \calC(t_0) $ has a neighborhood $U$ in $\Oomega\times
  (0,T)$ such that $\calC\cap U\subset C\times\{t_0\}$ and that $C$ is compact.
\end{proof}
\subsection{A characterization of the transversality conditions}
\label{ss:2.2}

\noindent
Lemma \ref{l:crucial-infinite-dim} sheds light onto the relation between
the  first two transversality conditions, and the surjectivity of the operator $\rmd \rmF (x_0,t_0)$.
In order to get a characterization of the \emph{full} set of  transversality conditions of
\berin Definition \erin \ref{ass:tran-infinity}, one has to
bring into play  the ``augmented'' operator
\begin{equation}
\label{G-operator}
\rmG: \Oomega \times (0,T) \times \X
\to \Y \times \Y \qquad \rmG(x,t,v) := (\rmF(x,t), \rmD_x \rmF(x,t)[v]).
\end{equation}
Let us first prove a preliminary result.

\begin{lemma}\label{rmk:G_indice_0}
If $\rmF: \Oomega \times (0,T) \to \Y$ satisfies Assumption \ref{F_ass1_infty},
then $\rmG(\cdot,t,\cdot) : \Oomega \times \X\to \Y \times \Y$
is a Fredholm map of index $0$ for every $t\in(0,T)$.
\end{lemma}

\begin{proof}
In order to simplify the notation, we omit
to indicate the explicit dependence
on the time variable.
For a fixed $(x,v)\in \Oomega\times\X$, we can write
\[
\rmd\rmG(x,v)[\tilde x,\tilde v]=(H[\tilde x],K[\tilde x]+H[\tilde y])
\quad\mbox{for every }(\tilde x,\tilde v)\in \X\times \X,
\]
where
\[
H:=\rmD_x\rmF(x),\qquad\qquad K:=\rmD_x^2\rmF(x)[v,\cdot].
\]
Thus the thesis follows if we show that for every
Fredholm operator $\berin H\in\erin\mathcal L(\X;\Y)$ of index $0$ and
for every $K\in \mathcal L(\X;\Y)$ the operator
\begin{displaymath}
 A:\mathcal L(\X\times \X;\Y\times \Y),\quad
 A[x_1,x_2]=(H[x_1],H[x_2]+K[x_1])
\end{displaymath}
is a Fredholm operator of index $0$.

We first observe that
\begin{displaymath}
  (x_1,x_2)\in \NN(A) \quad\Leftrightarrow\quad
  x_1\in \NN(H),\quad
  K[x_1]=-H[x_2],
\end{displaymath}
so that, by introducing the finite-dimensional space
\begin{displaymath}
  \mathsf M:=\Big\{x\in \NN(H):K[x]s\in \RR(H)
  =\big(\NN(H^*)\big)^\perp\Big\},
\end{displaymath}
it is easy to check that
\begin{equation}
  \label{eq:12}
  \dim\big(\NN(A)\big)=\dim \mathsf M\cdot\dim(\NN(H))
  <\infty.
\end{equation}
Let now $\SS$ be a finite-dimensional subspace of $\Y$
such that $\SS+\RR(H)=\Y$.
it is immediate to check that
$(\SS\times\SS)+\RR(A)=\Y\times \Y$ so
that $\codim(\RR(A))$ is finite and $A$
is a Fredholm operator.

We consider now the continuous perturbation
$[0,1]\ni\vartheta\mapsto A_\vartheta:=(H,H+\vartheta K)$:
every $A_\vartheta$ is a Fredholm operator so that
the index of $A=A_1$ coincides with the index of $A_0=
(H,H)$ since the index is a continuous function.
On the other hand it is \berin straightforward \erin to check that
$A_0$ has index $0$.
\end{proof}

Let us now consider the restriction
(still denoted by $\rmG$)
of $\rmG$ \berin from \eqref{G-operator} \erin to the open domain
\begin{equation}
  \label{eq:16}
  \Sigma:=\Oomega\times(0,T)\times\big(\X\setminus\{0\}\big)
\end{equation}
and notice that
\begin{equation}
  \label{eq:9}
  \rmG^{-1}(0,0)=\Big\{(x,t,v)\in
  \Sigma:  (x,t)\in \calS,\ v\in \NN\big(\rmD_x \rmF(x,t)\big)
  \Big\}.
\end{equation}
According to the common terminology, we say that
$(0,0)$ is a regular value of $\rmG$ if $\rmd \rmG$ is surjective
at each point of $\rmG^{-1}(0,0)$.

\begin{lemma}\label{l:G-(4)}
Let us suppose that $\rmF$ satisfies Assumption \ref{F_ass1_infty}
and let $\rmG:\Sigma\to \Y\times\Y$ be defined
as in \eqref{G-operator}, \eqref{eq:16}.
Then $\rmF$ satisfies the transversality conditions
if and only if $(0,0)$ is a regular value of $\rmG$.
\end{lemma}

\begin{proof}
A direct calculation shows that
$\rmd \rmG (x_0,t_0,v_0)[\tilde x,\tilde t,\tilde v]=(y_1,y_2)$
if and only if
\begin{equation}
\label{onto}
\left\{
\begin{array}{ll}
y_1 = \rmD_x \rmF (x_0,t_0)[\tilde x ] + \tilde t \partial_t \rmF
(x_0,t_0)=\rmd f(x_0,t_0)[\tilde x,\tilde t]
\\
y_2 = \rmD_x^2 \rmF (x_0,t_0)[v_0,\tilde x] +\tilde t\,
\rmD_x  \partial_t\rmF (x_0,t_0) [v_0] +
\rmD_x \rmF (x_0,t_0)[\tilde v ] .
\end{array}
\right.
\end{equation}
Let us first suppose that $(0,0)$ is
a regular value of $\rmG$,
let $(x_0,t_0)\in \calS$ and $0\neq v_0\in \NN\big(\rmD_x\rmF(x_0,t_0)\big)$.
It follows that
$(x_0,t_0,v_0)\in \rmg^{-1}(0,0)$ so that
$\rmd g(x_0,t_0,v_0)$ is onto
and, in view
of Lemma \ref{l:crucial-infinite-dim},
conditions (T1-2) hold.

Let us now choose $0\neq
\ell^*\in \NN\big(\rmD_x \rmF(x_0,t_0)^*\big)$,
$y_2\in \Y$ \berin such that $\pairing{\Y^*}{\Y}{\ell^*}{y_2}\neq0$, \erin
$y_1=0$, and a solution $(\tilde x,\tilde t,\tilde v)$ of \eqref{onto}.
From the first line of \eqref{onto} we get
\[
0  = \berin\pairing{\Y^*}{\Y}{\ell^*}{ \rmD_x \rmF (x_0,t_0)[\tilde x ] }  +
\tilde t \,\pairing{\Y^*}{\Y}{\ell^*}{ \partial_t \rmF (x_0,t_0)}
  =  0+ \tilde t \pairing{\Y^*}{\Y}{\ell^*}{ \partial_t \rmF (x_0,t_0)}\erin
\]
so that $\tilde t =0$, since
$\berin\pairing{\Y^*}{\Y}{ \ell^*}{\partial_t \rmF (x_0,t_0)}\erin\neq 0$.
The same \berin relation \erin
then yields $\tilde x\in \NN\big(\rmD_x \rmF(x_0,t_0)\big)$,
whence $\tilde x = \lambda v_0$, for some $\lambda \in \R$.

From the second line of \eqref{onto}
we obtain (taking into account that $\tilde t=0$)
\[
\begin{aligned}
0\neq\berin\pairing{\Y^*}{\Y}{ \ell^*}{ y_2}\erin   & =
\lambda\,\berin\pairing{\Y^*}{\Y}{ \ell^*}{ \rmD_x^2\rmF (x_0,t_0)[v_0,v_0]}   +
\pairing{\Y^*}{\Y}{\ell^*}{ \rmD_x \rmF (x_0,t_0)[\tilde v ] }\erin
\\
&=\lambda\,\berin\pairing{\Y^*}{\Y}{\ell^*}
{ \rmD_x^2 \rmF (x_0,t_0)[v_0,v_0 ] }\erin  + 0.
\end{aligned}
\]
Therefore $\berin\pairing{\Y^*}{\Y}{ \ell^*}{\rmD_x^2 \rmF (x_0,t_0)[v_0,v_0 ]}\erin\neq 0$,
which shows the validity of condition (T3).

\noindent Let us now prove the converse implication. We assume the
validity of (T1-2-3), we choose a point $(x_0,t_0,v_0)\in
\rmG^{-1}(0,0)$ and we want to prove that $\rmd \rmG(x_0,t_0,v_0)$
is onto; equivalently, if $(\ell^*_1,\ell^*_2)\in
\big(\RR(\rmd\rmG(x_0,t_0,v_0))\big)^\perp$ then
$\ell_1^*=\ell_2^*=0$.

The fact that $\ell_1^*=0$ is an immediate consequence
of the surjectivity of $\rmd \rmF(x_0,t_0)$, which is
the first component of
$\rmd\rmG(x_0,t_0,v_0)$ as showed by \eqref{onto}.

Choosing $\tilde x=0, \, \tilde t=0$
and an arbitrary $\tilde v$  in the second component
of \eqref{onto} we see that
$\ell_2^*\in \big(\RR(\rmD_x \rmF(x_0,t_0))\big)^\perp=
\NN\big(\rmD_x\rmF(x_0,t_0)^*\big)$.
Choosing now $\tilde v=0$, $\tilde t=0$, and $\tilde x=v_0$ we get
$\berin\pairing{\Y^*}{\Y}{\ell_2^*}{\rmD_x^2\rmF(x_0,t_0)[v_0,v_0]}\erin=0.$
Since $0\neq v_0\in \NN\big(\rmD_x\rmF(x_0,t_0)\big)$, (T3) yields $\ell_2^*=0$.
\end{proof}

\noindent
Notice that for $(x_0,t_0)\in \mathcal S$
the operator $\rmd \rmG(x_0,t_0,0)$ cannot be onto, since the second component of
\eqref{onto} reduces to $\rmD_x\rmF(x_0,t_0)[\cdot]$
whose range has codimension $1$.

\section{On the genericity of the transversality conditions}\label{s:3}

\noindent In this section we discuss the genericity of the
transversality conditions of \berin Definition \erin
\ref{ass:tran-infinity}. In the following Section \ref{ss:4.1}, we
give our main result in this direction, Theorem
\ref{thm:genericity_infty}. It states that, up to a \emph{small} (in
the topological sense) set of operators within a certain class, it
is always possible to perturb a (suitably smooth) map $\rmF :  \rbe \Oomega \rend
\times (0,T) \to \Y$ with such operators, in such a way as to obtain
a map $\tilde \rmF$ which complies with \berin Definition \erin
\ref{ass:tran-infinity}. \berin The proof of Theorem
\ref{thm:genericity_infty} relies on the characterization of the
transversality conditions provided in Section \ref{ss:2.2}, and on a
well-known genericity result from the seminal paper
\cite{Saut-Temam79}. In Section \ref{ss:4.2} we revisit the
genericity result,  as well as Theorem \ref{prop:isolated_infty}, in
the case where $\rmF$ is the space differential of a time-dependent
energy functional.\erin

 Notice that the genericity of conditions (T1-T2) of Definition
\ref{ass:tran-infinity} with respect to a simpler class of
perturbations is a direct consequence of the results of
\cite{Saut-Temam79}.

\subsection{The genericity results}
\label{ss:4.1}

\newcommand{\JJ}{j}
In order to make precise in which sense we are going to prove that the transversality conditions hold
``generically'', let us recall that a set $N$ in a topological space $\mathcal T$
is said to be \emph{nowhere dense} if the interior of its closure is empty.
Equivalently, its complement $\mathcal T\setminus A$
has dense interior, i.e.~$A$ is contained in the complement
of an open and dense set.

A set is said to be \emph{meagre} if it is contained in a countable union of nowhere dense  sets.
Conversely, a set is \emph{residual} if it is the complement of a
meagre set, i.e.~it contains the intersection
of a countable collection of open and dense sets.

We will suppose that
\begin{equation}
\text{$\X,\Y,\Z$ are separable Banach spaces,}\quad  X\subset \X,\
0\in Z\subset \Z\quad\text{open and connected}, \label{eq:13}
\end{equation}
and we will deal with a sufficiently large class of  additive
perturbations of  \berin the map $f: \Oomega \times (0,T) \to \Y$ of
the type
\begin{equation}
  \label{eq:11}
  \tilde f(x,t):=f(x,t)+y+j(x,z),\quad
  y\in \Y,\ z\in Z,
\end{equation}
obtained by means of
$\JJ : \Oomega\times Z \to \Y$ of class $\rmC^2$
with $\JJ(x,0)=0$.
We will suppose that
for every $x\in \Oomega,\, t\in (0,T),\, z\in Z$ and $v\in \X\setminus\{0\}$
the following admissibility conditions are satisfied by $\JJ$:
\begin{gather}
  \label{eq:15-bis}
    \tag{J1}
    \begin{gathered}
      \text{$\rmD_x \tilde f(x,t)=
        \rmD_x f(x,t)+\rmD_x \JJ(x,z)$ is a Fredholm operator in $\calL(\X;\Y)$,}
     \end{gathered}
    \\
    \label{eq:15-ter}
    \tag{J2}
    (\tilde v,\tilde z)\mapsto \rmD_x \tilde f(x,t)[\tilde v]+\rmD^2_{xz}
    \JJ(x,z)[v,\tilde z]\text{ is surjective,}
\intertext{i.e.~ for every $w\in \Y$ there exist $\tilde v,\tilde z\in
\X\times \Z$ such that }
\label{eq:24}
\rmD_x f(t,x)[\tilde v]+\rmD_x \JJ(x,z)[\tilde v]+
  \rmD^2_{xz}\JJ(x,z)[v,\tilde z]=w.
\tag{J2'}
\end{gather}
Notice that when $\JJ$ is a bilinear map, \eqref{eq:24}
takes the simpler form
\begin{equation}
  \label{eq:25}
  \rmD_x f(t,x)[\tilde v]+ \JJ(\tilde v,z)+
  \JJ(v,\tilde z)=w.
\end{equation}
In Example \ref{ex:admissible-J} below, we exhibit a particular case
of admissible mapping \rbe $j$, \rend with values in a suitable
space of compact operators.
\begin{example}
\label{ex:admissible-J}
\upshape
Let us consider a
separable
    Banach space $\calK$ continuously included in $\calL(\X;\Y)$
 such that
\begin{gather}
  \label{eq-15K2}
    \tag{K1}
    \text{every $K\in \calK$ is a compact operator,}
    \\
    \label{eq-15K3}
    \tag{K2}
    \text{for every $v\in\berin\X\setminus\{0\}\erin,\,w\in \Y$
      there exists $K\in \calK:\quad K[v]=w$.}
\end{gather}
As a typical example, we can choose $\mathcal{K}$  to be the closure in
$\calL(\X,\Y)$ of the
nuclear operators of the form
\begin{displaymath}
  K[x]=\sum_{n=1}^N\pairing{\X^*}\X{\ell_n}x\,y_n,\quad
  \ell_n\in \X_0^*,\ y_n\in \Y.
\end{displaymath}
where $\X_0^*$ is a separable subspace of $\X^*$ that separates the
points of $\X$.
Observe that condition \eqref{eq-15K3}  holds. Indeed, whenever $v\in \X \rbe \setminus \{0\} \rend,\ w\in \Y$ are given,
by choosing $\ell\in \X_0^*$ so that
$\pairing{\X^*}\X\ell {v}=1$ we can simply set
\begin{equation}
  \label{eq:18}
  K[x]:=\pairing{\X^*}\X\ell x\,w.
\end{equation}

Hence, we let $\Z:= \mathcal{K}$
and take as $\JJ$ the bilinear map
\begin{equation}
\label{citato-dopo}
\JJ:  \X\times \mathcal{K} \to  \Y,\quad
\JJ(x,K):=K[x]
\end{equation}
Then, condition \eqref{eq:15-bis} is satisfied,  since for every
$K\in \calK$ the differential $\rmD_x\JJ(x,K)=K$ is a compact
operator  and  thus, when added to a Fredholm operator, gives rise
to  a Fredholm operator with the same index. It is immediate to
check that \eqref{eq-15K3} guarantees the validity of
\eqref{eq:15-ter},  since $\rmD^2_{xz}\JJ(x,K)[\tilde v,\tilde
K]=\tilde K[\tilde v]$.
\end{example}
\erin

We then have the following theorem.
\begin{theorem}\label{thm:genericity_infty}
Let  us assume that \eqref{eq:13} and the admissibility conditions
{\em (J1-2)} hold, for a map $\rmF \in \rmC^3(\Oomega \times (0,T);
\Y)$ complying with Assumption \ref{F_ass1_infty} \rbe and for $\JJ
\in \rmC^2(\Oomega\times Z ; \Y)$. \rend

Every open neighborhood $U$ of the origin in
$\Y\times \Z$ contains a residual subset $U_r$
such that for every $(y,z)\in U_r$ the map
\begin{equation}\label{eq:perturbation_F}
\tilde\rmF:\X\times(0,T)\longrightarrow \Y \qquad\tilde\rmF(x,t):=
\rmF(x,t)+y+ \JJ(x,z).
\end{equation}
satisfies the transversality conditions \berin of Definition \erin \ref{ass:tran-infinity}.
\end{theorem}

\begin{proof}
First of all, it is useful to pass from $\rmF$ to the functional \berin
\[
\mathcal{F}: \Oomega \times (0,T)\times \Y\times
\Z
 \longrightarrow \Y\qquad
\mathcal{F}(x,t,y,z):=\rmF(x,t) + y +
\JJ(x,z),
\]
which incorporates the perturbation terms, so that
the map $\tilde\rmF$ of \eqref{eq:perturbation_F}
coincides with
$\calF(\cdot,\cdot,y,z)$.

\berin In accord with  \eqref{G-operator}, \erin 
we consider the \berin \emph{augmented} \erin functional
$\mathcal{G}: \Sigma\times \Y \times Z \to \Y \times \Y$
\berin (recall that  $\Sigma:=\Oomega\times(0,T)\times\big(\X\setminus\{0\}\big)$) \erin
\begin{equation}
\label{augmen-G}
\begin{gathered}
\mathcal{G}(x,t,v,y,z):= \left(
\begin{array}{cc}
\mathcal{F}(x,t,y,z)\\
\rmD_x \mathcal{F}(x,t,y,z)[v]
\end{array}
\right)
=
\left( \begin{array}{cc}
\rmF(x,t) + y + \JJ(x,z)\\
\rmD_x \rmF (x,t) [v] + \rmD_x \JJ(x,z)[v]
\end{array}
\right),
\end{gathered}
\end{equation}
which for every $(y,z)\in \Y\times Z$ gives raise to
the perturbed functionals
\begin{equation}
  \label{eq:17}
  \tilde g(x,t,v)=\calG(x,t,v,y,z)=
  \Big(f(x,t)+y+\JJ(x,z),\big(\rmD_x f(x,t)+\rmD_x \JJ(x,z)\big)[v]\Big).
\end{equation}
\erin
By Lemma \ref{l:G-(4)}, $\tilde f$ satisfies the transversality
conditions if and only if $(0,0)$ is a regular value for
$\tilde g$. We conclude by applying the next result.
\end{proof}

\begin{proposition}
\label{analogue-1-true}
Under the same assumptions of \berin Theorem \erin \ref{thm:genericity_infty}, the set
\begin{equation*}
\berin \mathfrak{R}:= \{ (y,z) \in  \Y \times \Z \, :
 \ (0,0) \text{ is a regular value of  the map }
 (x,t,v)\mapsto \mathcal{G}(x,t,v,y,z)\} \erin
\end{equation*}
is \emph{residual} in  \berin $\Y\times \Z$. \erin
\end{proposition}
\begin{proof}
We are going to apply \cite[Theorem 1.1, Remark A.1]{Saut-Temam79} and
thus check that the corresponding assumptions hold, namely \berin
\begin{enumerate}[(a)]
\item the space   $\Y \times \Z $ is separable;
\item $\mathcal{G}\in\rmC^2(\Sigma \times \Y \times Z;\Y\times \Y)$ and
  for every $(y,z) \in  \Y \times Z$ the map $(x,t,v) \mapsto \mathcal{G}(x,t,v,y,z)$
is Fredholm with index $1$;
\item $(0,0)$ is a regular value for $\mathcal{G}$.
\end{enumerate}
Condition (a) clearly holds since we \berin have \erin assumed
$\Y$ and $\Z$ separable.  \erin

As for (b), clearly $\mathcal{G}$ is of class $\rmC^2$ thanks to the
enhanced $\rmC^3$-regularity required of $\rmF$ \berin and to the
$\rmC^2$-regularity of $\JJ$.  \erin Moreover, \rbe condition (J1),
\rend Lemma \ref{rmk:G_indice_0} and the same argument of Remark
\ref{rmk:index_0_index_1} yield that every perturbed functional
$\tilde\rmG$ as in \eqref{eq:17} is Fredholm of index $1$.

Let us now focus on the last property (c),  and let us set
$h(t,x,z):=f(t,x)+\JJ(z,x)$.  We have to check that, \berin if
$\mathcal{G}(x_0,t_0,v_0,y_0,z_0) =(0,0)$, then
$\rmd \mathcal{G}(x_0,t_0,v_0,y_0,z_0)$ is onto, 
namely that for every
$(w_1, w_2) \in \Y \times \Y$
there exists $(\tilde
x, \tilde t, \tilde v, \tilde{y}, \tilde z) \in \X \times \R
\times \X \times \Y \times \Z$
such that
\begin{subequations}
\label{surj}
\begin{align}
 &
 \label{w-1} w_1= \rmD_x h(x_0,t_0,z_0) [\tilde x] +
 \tilde t\,
\partial_t \rmF(x_0,t_0) +\tilde y +
\rmD_z \JJ(x_0,z_0)[\tilde z],
\\
& \label{w-2}  w_2 = \rmD_x^2 h(x_0,t_0,z_0) [v_0,\tilde x]
  +
\tilde t\,
\partial_t \rmD_x f(x_0,t_0) [v_0]+
\rmD_x h(x_0,t_0,v_0) [\tilde v]+
\mathrm{D}^2_{xz}\JJ(x_0,z_0)[v_0,\tilde z].
\end{align}
\end{subequations}
For this, we choose $\tilde x =0$, $\tilde t =0 $ and we use
condition \rbe (J2) \rend to find $\tilde v\in \X$ and $\tilde z \in
\Z$ such that \eqref{w-2} is satisfied.
In order to fulfill   \eqref{w-1},
 we take
$\tilde y:=w_1- \mathrm{D}\JJ(z_0)(\tilde z)[x_0]$. With this, we conclude that (c) holds. \erin
\end{proof}

\begin{remark}[The finite-dimensional case]
  \label{rem:fdimproof}
  \upshape
  \berin Our genericity result in the finite-dimensional setting of
  Sec.\ \ref{s:1},  Theorem
  \ref{thm:genericity_finite},  derives from Theorem
  \ref{thm:genericity_infty}.

   Indeed,   
the perturbed map $\tilde\rmF$ in
  \eqref{eq:perturbation_F-finite}  is of the form
  \eqref{eq:perturbation_F}, where we have taken as admissible perturbation $\JJ$
  the mapping \eqref{citato-dopo} from Example \ref{ex:admissible-J}.  \erin
  We accordingly introduce the
  finite-dimensional analogues of the functionals $\mathcal{F}$ and
  $\mathcal{G}$, to which Lemma \ref{l:G-(4)} clearly applies.  In
  this case, the set $\Sigma$ defined in \eqref{eq:16}
  reduces to
  $\Oomega\times(0,T)\times(\R^n\setminus\{0\})$
  and $\calK $ is $\mathbb{M}^{n \times n}$.
  Also, note that since
  $\rmF \in\rmC^3 (\Oomega \times (0,T); \R^n)$,
  then $\mathcal G$ is
  of class $\rmC^2$.  Proceeding as in the proof of
  \cite[Theorem 1.1]{Saut-Temam79} it is possible to show that proving
  that the set
  \begin{equation}\label{eq:conclu_finite}
    \begin{gathered}
      \mathfrak{R}:= \{ (y,K ) \in \R^n\times \mathbb M^{n\times n}
      \, : (0,0) \text{ is a regular
        value of the map } \mathcal{G}(\cdot,\cdot,\cdot,y,K) \}
    \end{gathered}
  \end{equation}
  has full Lebesgue measure,
  is equivalent to proving that the set
  $\mathfrak V$ of the regular
  values of the function $\pi:\mathcal G^{-1}(0,0)\to
  \R^n\times \mathbb M^{n\times n}$ has
  full Lebesgue measure in $\R^n\times \mathbb M^{n\times n}$,
  where $\pi$ is the projection onto the last two components
  in $\Sigma\times \R^n\times \mathbb M^{n\times n}$.
  This property follows from the regularity
  of $\mathcal G$,
  the fact that $(0,0)$ is a regular value of
  $\calG$ so that $\pi$ is a $\rmC^2$ map, and the
  classical Sard's Theorem, since
  $2>{\rm dim}(\mathcal
  G^{-1}(0,0))-{\rm dim}(\R^n\times \mathbb M^{n\times n})
  =n^2+n+1-(n+n^2)=1$.  Once
  (\ref{eq:conclu_finite}) is established, we can conclude as done in
  the proof of Theorem \ref{thm:genericity_infty}.
\end{remark}


\subsection{Critical points of an energy functional}
\label{ss:4.2}


In this last section we consider the particular
case when
\[
\text{$\X\subset \H,\  \Y\subset \H^*$
\quad with
 continuous and dense inclusions,}
\]
$\H$ is a separable Hilbert space and
$\rmF$ is the space differential of a time-dependent
functional $\mathscr E:\Oomega\times (0,T)\to\R$, i.e.
\begin{equation}
  \label{eq:19}
  \rmF=\rmD_x \mathscr E,\qquad
  \mathscr E\in \rmC^3(\Oomega\times(0,T)),\qquad
  \Oomega\subset \X\quad\berin\text{connected and open.}\erin
\end{equation}
 We are thus assuming that the differential of the energy takes
values (and is regular) in a possibly smaller Banach space $\Y$
contained in $\X^*$. On the other hand (see \cite[Remark
1.1]{Saut-Temam79}) for every $(x,t)\in \Oomega\times (0,T)$ we will
suppose that the linear operator $L$ associated with the second
order differential $\rmD^2_x\mathscr E$ admits a unique continuous
extension $\tilde L\in \calL(\H,\H^*)$ satisfying
\begin{equation}\label{eq:26}
  D(\tilde L;\Y)=\big\{h\in \H:\tilde L h\in \Y\big\}\subset \X.
\end{equation}
Notice that for every $v,w\in X$
\begin{displaymath}
  \pairing{\H^*}\H {Lv}w=
  \rmD^2_x \mathscr E(x,t)[v,w]=
  \rmD^2_x \mathscr E(x,t)[w,v]=
 \rbe  \pairing{\H^*}\H {Lw}v, \rend
\end{displaymath}
so that $\tilde L$ is selfadjoint.

\rbe In this setting, \rend $\calC$ is the set of critical points of
$\mathscr E$ and $\calS$ is the corresponding singular subset
\begin{equation}
  \label{eq:1}
  \calC:=\Big\{(x,t)\in \Oomega\times (0,T):
  \rmD_x \mathscr E(x,t)=0\Big\},
  \quad
  \calS:=\Big\{(x,t)\in \calC:\rmD_x^2\mathscr E(x,t)\
  \text{is not invertible.}\Big\}.
\end{equation}
We will assume that $\rmD_x \mathscr E$ is a Fredholm map  of index
$0$. In particular we can identify the kernel
$N=\NN(\rmD^2_x\mathscr E(x,t))$ in $\X$ with
$N^*=\NN(\rmD^2_x\mathscr E(x,t)^*)$ in $\Y^*$: it is in fact easy
to check that the canonical inclusion $\X\subset \Y^*$ induced by
the scalar product of $\H$ yields $N\subset N^*$, since \rbe (still
using the notation $L$ for the second order differential
$\rmD^2_x\mathscr E$), \rend  $Lv=0$ and $v\in \X$ yield
\begin{displaymath}
 \rbe  \pairing{\X^*}\X {L^* v}w=\pairing\H{\H^*}v{Lw}
  =  \pairing{\H^*}{\H}{Lv}w=0\quad\text{if }v\in \NN(L). \rend
\end{displaymath}
\rbe This implies that  $\NN(L)$ and  $\NN(L^*)$ coincide, \rend
since they have the same dimension by the index property.

The transversality conditions read 
\begin{definition}[Transversality conditions for
  a time-dependent functional]
  \label{ass:tran-E}
  We say that \berin $\scrE$ satisfies at a point $(x_0,t_0) \in \calS$
 the transversality conditions if \erin
\begin{enumerate}[\rm (E1)]
\item
  $\mathrm{dim}(\NN (\rmD^2_x \scrE (x_0,t_0)))=1$;
\item
  If $0\neq v\in \NN(\rmD^2_x \scrE(x_0,t_0))$
  then $\pairing{\X^*}{\X}{\partial_t
    \rmD_x \scrE (x_0,t_0)} {v}
  \neq 0$;
\item
  If $0\neq v\in \NN(\rmD^2_x \scrE(x_0,t_0))$ then
  $
  \rmD_x^3\scrE (x_0,t_0)[v,v,v]
    \neq 0$.
\end{enumerate}
The functional $\scrE$ satisfies the transversality conditions
if {\em (E1-2-3)} hold for every $(x_0,t_0)\in \calS$.
\end{definition}

Theorem \ref{prop:isolated_infty} immediately yields:

\begin{corollary}
  If $\scrE\in \rmC^3(\Oomega\times(0,T))$
  is a time-dependent functional
  with Fredholm differentials $\rmD\scrE_x(\cdot,t)$
  which satisfies the transversality conditions
  {\em (E1-2-3)}, then
  the sets $\calC(t):=\big\{x\in \Oomega:\rmD_x\scrE(x,t)=0\big\}$ are discrete
  for every $t\in (0,T)$.
\end{corollary}
\berin We now address the
 genericity of
the transversality conditions from Definition \ref{ass:tran-E}. \erin
\berin In Corollary  \ref{cor-energy} below we  rephrase, in terms
of the functional $\scrE$, the statement of   Theorem
\ref{thm:genericity_infty},  considering here only the simple case
of \ Example \ref{ex:admissible-J}.  Accordingly, \erin we consider
the set $\calN_{sym}$ obtained by taking the closure in
$\calL^2(\X;\R)$ of all the symmetric bilinear forms of the type
\begin{equation}
  \label{eq:20}
  \mathscr K(x,y)=\sum_{j=1}^n
  \paired {\ell_j} x\,\paired {\ell_j} y.
\end{equation}
\begin{corollary}
\label{cor-energy}
  Let $\scrE\in \rmC^4(\Oomega\times(0,T))$ be
  a time-dependent functional with
  Fredholm differentials.
  Every open neighborhood $U$ of the origin in
  $\X^*\times \calN_{sym}$ contains a residual
  subset $U_r$ such that
  for every $(\ell,\mathscr K)\in U_r$ the functionals
  \begin{equation}
    \label{eq:21}
    \tilde \scrE(x,t)=\scrE(x,t)+\paired \ell x+
    \frac 12\mathscr K(x,x)
  \end{equation}
  satisfy the transversality conditions
  {\em (E1-2-3)}.
\end{corollary}
\begin{proof}
  We apply Theorem \ref{thm:genericity_infty}:
  notice that the perturbations \eqref{eq:21}
  of $\scrE$ correspond to the family
  of perturbations for $\rmF=\rmD_x\scrE$
  \begin{equation}
    \label{eq:22}
    \tilde \rmF(x,t)=\rmF(x,t)+\ell+K[x],
  \end{equation}
  where $K\in \calL(\X,\X^*)$ is associated with $\mathscr K$ by
  \begin{equation}
    \label{eq:23}
    \paired {K\,x_1}{x_2}=\mathscr K(x_1,x_2),\quad
    \text{ for every}\quad x_1,x_2\in \X.
  \end{equation}
  Clearly, the collection of all such operators
  satisfies the admissibility conditions (K1-2);
  in order to check (K3), we fix
  $\berin x\in\X\setminus\{0\}\erin$, $\ell\in \X^*$, and we consider the bilinear forms
  \begin{align*}
    \mathscr K(x_1,x_2)&=
    \frac{\paired{\ell}{ x_1} \paired{ \ell}{ x_2}  }{
      \paired{\ell }{ x}  } &&  \text{ if } \paired{\ell}{x}
\neq 0,
\\
\mathscr K(x_1,x_2)&=
 \paired{\ell}{ x_1}  \paired{ x^*}
  {x_2}  + \paired{ x^*}{ x_1}   \paired{ \ell}
  {x_2}   & & \text{ if } \paired{\ell}{ x} =0,
  \end{align*}
  where $x^*\in \X^*$ satisfies
  $\paired {x^*}x=1$.
  It is immediate to check that the associate linear operator
  $K$ satisfies $K[x]=\ell$.
\end{proof}

 We  conclude  by  exhibiting an integral
energy functional $\mathscr{E}$, whose critical points (i.e.\ the zeroes of its space differential)
solve a semilinear elliptic equation on a bounded domain $\Omega$. Therefore, as customary we will use  the letter $x$ to denote the points in $\Omega$, and write
$\mathscr{E}(u,t)$ in place of $\mathscr{E}(x,t)$.
\begin{example}
\label{ex:concrete-pde}
\upshape
Let $\Omega $ be a  bounded connected open set  in  $\R^d$, $d\leq 3$, and let
\[
\X = H^2(\Omega) \cap H_0^1 (\Omega), \qquad  \Y = L^2(\Omega),
\]
and
\[
\mathscr{E}(u,t):= \int_\Omega \left( \frac12 |\nabla u(x)|^2 +\mathcal{W}(u(x)) - \ell(t) u(x) \right) \, \dd x
\]
with \rbe $\ell \in \mathrm{C}^4(0,T; L^2(\Omega))$ \rend and $\mathcal{W}$ the usual double-well potential
$\mathcal{W}(u) = (u^2-1)^2/4$.
Observe that $\mathscr{E} \in \mathrm{C}^4(\X \times (0,T))$ thanks
to the continuous imbedding of $\X$ in $L^\infty(\Omega)$.

We have that
\[
f= \rmD_u \mathscr{E}: \X \times (0,T) \to \Y
\text{ is given by }
f(u,t)  = Au + \mathcal{W}'(u) - \ell(t)
\]
with $A:  H^2(\Omega) \cap H_0^1 (\Omega) \to L^2(\Omega)$ the Laplacian operator with homogeneous Dirichlet boundary conditions. Note that $A$ is Fredholm with index $0$.
It follows from \cite[Thm.\ 5.26, p.\ 238]{Kato} that also $f(\cdot, t)$ is a   Fredholm map with index $0$
for every $t \in (0,T)$. We have that
\[
\rmD_u f(u_0,t_0) [v] = Av + \mathcal{W}{''}(u_0) v \qquad \text{ for all } v \in \X.
\]

Let us now construct an explicit perturbation of $f$ to which Theorem \ref{thm:genericity_infty} applies.
We take
\begin{equation}
\label{multipl-J}
\Z =\rmC(\overline\Omega) \ \text{ and } \  \JJ : \Z \times \X
\to \Y \text{ defined by }
\JJ(u,z):= zu,
\end{equation}
so that $\JJ(u,z)$ is a bilinear operator.
Observe that the perturbed field $\tilde f(u,t) =  f(u,t) + y + zu $
 is the space differential of the perturbed energy
 \begin{equation}
 \tilde {\mathscr{E}} (u,t) = \mathscr{E} (u,t) +
 \int_\Omega y\, u\,\dd x +\frac12 \int_\Omega z \,u^2 \dd x.
 \label{eq:27}
\end{equation}
 It follows again from \cite[Thm.\ 5.26, p.\ 238]{Kato} that
 $\tilde f(\cdot,t)$ is a Fredholm map of index $0$
  for every $t \in (0,T)$.

 In order to check that Theorem \ref{thm:genericity_infty} applies in this setting, it would remain to verify
 that $\JJ$ in \eqref{multipl-J} complies with  condition (J2')
 in the form \eqref{eq:25}.
 Therefore, for given $u_0\in\X, t_0\in (0,T), v_0\in \X\setminus\{0\}$ and
 $z_0\in \rmC(\overline\Omega)$
 we want to show that
 for every $w\in L^2(\Omega)$ the equation
 \begin{equation}
 \label{suff-check}
 A \tilde v + (W{''}(u_0)+z_0) \tilde v + v_0 \tilde z=w
 \end{equation}
 admits at least a solution $\tilde v\in \X$, $\tilde z\in \rmC(\overline
 \Omega)$.
Since the operator
$\tilde v\mapsto A\tilde v+  (W{''}(u_0)+z_0) \tilde v$
is Fredholm of index $0$  from $\X$ to $\Y$
(and thus its range is closed with finite codimension),
it is sufficient to prove that the only element $\xi\in L^2(\Omega)$
such that
\begin{displaymath}
  \int_\Omega \xi\Big(A \tilde v + (W{''}(u_0)+z_0) \tilde v + v_0
  \tilde z\Big)
  \,\dd x=0\quad\text{for every }\tilde v\in \X,\ \tilde z\in \rmC(\overline\Omega)
\end{displaymath}
is $\xi\equiv 0$.

Choosing $\tilde z=0$ and an arbitrary $\tilde v$ we
get
\begin{displaymath}
  \int_\Omega \xi\Big(A \tilde v + (W{''}(u_0)+z_0) \tilde v
  \Big)
  \,\dd x=0\quad\text{for every }\tilde v\in \X,
\end{displaymath}
so that $\xi\in \X$ and
 \begin{equation}
 \label{solve}
 A \xi + (W{''}(u_0)+z_0) \xi=0.
 \end{equation}

 On the hand, choosing
 $\tilde v=0$ and arbitrary $\tilde z$ we get
 $\xi=0$ on the set $\omega=\{x\in \Omega: v_0(x)\neq0\}$,
 which is open and not empty since $v_0\in \rmC(\overline\Omega) \neq
 0$.
 Then, \eqref{solve} and  the unique continuation principle (see e.g.\ \cite{Jerison-Kenig})
  imply that $\xi =0$. Therefore, condition
 \eqref{suff-check}
  is satisfied
  \begin{theorem}
    Every open neighborhood $U$ of the origin in
    $L^2(\Omega)\times \rmC(\overline\Omega)$
    contains a (dense) residual set $U_r$ such that
    the functional $\tilde{\mathscr E}$
    defined by \eqref{eq:27} satisfies the transversality conditions
    for every $(y,z)\in U_r$.
    In particular, for every $t\in (0,T)$ and $(y,z)\in U_r$ the
    solutions $u\in H^2(\Omega)\cap H^1_0(\Omega)$ of the
    equation
    \begin{equation}
      \label{eq:28}
      -\Delta u+\calW'(u)+zu=y+\ell(t)\quad\text{in }\Omega,\quad
      u=0\quad\text{on }\partial\Omega,
    \end{equation}
    are isolated in $H^2(\Omega)\cap H^1_0(\Omega)$.
  \end{theorem}

\end{example}



\bibliographystyle{amsalpha}

\begin{thebibliography}{GMT80}


\bibitem{ars}\textsc{V.~Agostiniani, R.~Rossi, G.~Savar\'e},
           {\em Balanced Viscosity solutions of
             singularly perturbed gradient flows in infinite dimension},
            in preparation.

\bibitem{virginia}\textsc{V.~Agostiniani},
            {\em Second order approximations of quasistatic evolution problems in finite dimension},
            Discrete Contin. Dyn. Syst. A {\bf 32} n. 4, 2012, 1125--1167.

\bibitem{Amb_Pro}
     \newblock\textsc{A.~Ambrosetti, G.~Prodi},
     \newblock ``A primer of nonlinear analysis,"
     \newblock Cambridge University Press, Cambridge, 1995.


\bibitem{guck-holmes} \textsc{J.\ Guckenheimer,  P.\ Holmes},``Nonlinear oscillations, dynamical dystems and
bifurcations on vector fields'', Applied Mathematical Sciences, 42, Springer-Verlag, New York, 1983.

\bibitem{haragus} \textsc{M.\ Haragus, G.\ Iooss}, ``Local bifurcations, center manifolds,
and normal forms in infinite-dimensional dynamical dystems'',
Springer-Verlag London,  London, 2011.

\bibitem{Jerison-Kenig} \textsc{D.\ Jerison, C.E.\ Kenig}, {\em  Unique continuation and absence of positive eigenvalues for Schršdinger operators},
Ann. of Math. (2) {\bf 121}, 1985,  463--494.

\bibitem{Kato}
\textsc{T.\ Kato}, ``Perturbation theory for linear operators'',
(Reprint of the 1980 edition), Classics in Mathematics.
Springer-Verlag, Berlin, 1995.


\bibitem{Saut-Temam79}\textsc{J.~C.~Saut, R.~Temam},
   {\em Generic properties of nonlinear boundary value problems},
   Comm. Partial Differential Equations \textbf{4},  1979, 293--319.

\bibitem{Smale}\textsc{S.~Smale},
      {\em An infinite dimensional version of Sard's theorem},
     Amer. J. Math. {\bf 87}, 1965, 861--866.

\bibitem{Sot}\textsc{J.~Sotomayor},
        {\em Generic bifurcations of dynamical systems},
         Dynamical systems (Proc. Sympos., Univ. Bahia, Salvador, 1971), Academic Press, New York, 1973, 561--582.

\bibitem{vanderbauwhede} \textsc{A.\  Vanderbauwhede}, ``Center manifolds, normal forms and elementary bifurcations'', Dynamics
reported, Vol. 2, 89–169, Dynam. Report. Ser. Dynam. Systems Appl., 2, Wiley, Chichester,
1989.

\bibitem{zanini}\textsc{C.~Zanini},
            {\em Singular perturbations of finite dimensional gradient flows},
            Discrete Contin. Dyn. Syst. \textbf{18}, 2007, 657--675.

\end{thebibliography}

\end{document}